\documentclass[psamsfont,12pt]{amsart}
\usepackage[usenames]{xcolor}
\usepackage{amsmath,amssymb}
\usepackage{amsthm,amsfonts,amsmath,amssymb,latexsym,epsfig,mathrsfs,marvosym}

\usepackage{hyperref}
\usepackage[english]{babel}

\DeclareMathAlphabet\oldmathcal{OMS}        {cmsy}{b}{n}
\SetMathAlphabet    \oldmathcal{normal}{OMS}{cmsy}{m}{n}
\DeclareMathAlphabet\oldmathbcal{OMS}       {cmsy}{b}{n}
\usepackage{hyperref}
\usepackage[english]{babel} 

\newtheorem{theorem}{Theorem}[section]
\newtheorem{lemma}[theorem]{Lemma}

\newtheorem{proposition}[theorem]{Proposition}
\newtheorem{notation}[theorem]{Notation}

\newenvironment{example}{\medskip \refstepcounter{theorem}
\noindent  {\bf Example \thetheorem}.\rm}{\,}

\usepackage{calc}
\newsavebox\CBox
\newcommand\hcancel[2][0.5pt]{%
  \ifmmode\sbox\CBox{$#2$}\else\sbox\CBox{#2}\fi%
  \makebox[0pt][l]{\usebox\CBox}%
  \rule[0.5\ht\CBox-#1/2]{\wd\CBox}{#1}}
  \usepackage{tikz}
\usetikzlibrary{calc}

\theoremstyle{remark}
\newtheorem{rem}[theorem]{Remark}
\theoremstyle{definition}

\newcommand{\intprod}{\mathbin{\raisebox{\depth}{\scalebox{1}[-1]{$\lnot$}}}}

\newcommand{\nts}[1]{\marginpar{#1}}
\renewcommand{\nts}[1]{}
\newcommand{\R}{\mathbf{R}}

\def\bfZ{\mbox{{\bf Z}}}

\def\DF{\mbox{{\bf DF}}}

\def\mF{\mathcal{F}}      \def\mB{\mathcal{B}}  \def\mL{\mathcal{L}}      \def\mN{\mathcal{N}}  \def\mE{\mathcal{E}}  \def\mX{\mathcal{X}}    \def\mO{\mathcal{O}}  \def\mA{\mathcal{A}}

\def\metric{\Omega}
\def\euler{{\bf{e}}}
  
 \def\bD{\mathbb D}  \def\bC{\mathbb C} \def\bR{\mathbb R}       \def\bP{\mathbb P}

\def\R{\mathbb R}\def\C{\mathbb C}

\def\fract#1#2{\raise4pt\hbox{$ #1 \atop #2 $}}

\def\kt{\mathfrak{t}}

\def\ra{\rightarrow }
\def\bfZ{{\bf Z} }

\def\valueMOM{{\bf h}}
\def\mom{{ \mu}}

\def\b{b}
\def\excep{E}
\def\bu{Z}
\def\classO{\delta}
\def\quot{{\bf q}}
\def\w{\tilde{{\bf w}}}
\def\we{{\bf w}}

\def\metric{{\Omega}}
\def\cG{G}
\def\bm{p}

\def\cst{\overline{c}}

\def\biho{\psi}

\def\A{\alpha}

\def\B{\beta}
\begin{document}

\title[Localizing the Donaldson--Futaki invariant]{Localizing the Donaldson--Futaki invariant}
\date{\today}

\author{Eveline Legendre}

 \address{Eveline Legendre\\ Universit\'e Paul Sabatier\\
 Institut de Math\'ematiques de Toulouse\\ 118 route de Narbonne\\
 31062 Toulouse\\ France}
\email{eveline.legendre@math.univ-toulouse.fr}

 \thanks{The author was partially supported by France ANR project EMARKS No ANR-14-CE25-0010. }

\maketitle

\begin{abstract}
We use the equivariant localization formula to prove that the Donaldson--Futaki invariant of a compact smooth (K\"ahler) test configuration coincides with the Futaki invariant of the induced action on the central fiber when this fiber is smooth or have orbifold singularities. We also localize the Donaldson--Futaki invariant of the deformation to the normal cone.
\end{abstract}

\section{Introduction}

 The Yau--Tian--Donaldson conjecture has been a central theme of K\"ahler geometry in the last 30 years. In one side of this conjectured correspondance, one tests the K--stability of a K\"ahler manifold $(X,[\omega])$ using the {\it Donaldson--Futaki invariant} of tests configuration over it.   

This invariant has a long history, the first candidate was the {\it generalized} Futaki invariant of the central fiber as suggested by Tian~\cite{Tian}. The work of Tian was motivated by Yau conjecture about K\"ahler--Einstein metrics~\cite{yau} and he was working in the Fano context but the definition he suggested does not need this hypothesis. Then Donaldson reinterpreted and generalized this invariant in the polarized case, in terms of the coefficients of the Hilbert series describing the asymptotics expansion of the dimension and weights of the action on the space of sections of the central fiber \cite{don:scalar}. Starting from this generalization, Odaka~\cite{odaka} and Wang~\cite{wang}, see also \cite[p.315]{don:scalar}, exhibited an intersection product formulation of this now called Donaldson--Futaki invariant. One gain of this formulation is a direct interpretation in the non-polarized/transcendental case as exploited in \cite{DervanRoss,zakarias}.

In this note, we start with the intersection product formulation of the Donaldson--Futaki invariant $\DF$ of a compact smooth test configuration $(\mX,[\Omega])$ over a smooth compact K\"ahler manifold as a definition. We will recall the precise definition of a compact K\"ahler test configuration in Section~\ref{subsectTCDF0} and prove the following : we observe that $\DF(\mX,[\Omega])$ is the intersection of $S^1$--equivariant closed forms on $\mX$ and show that there exists smooth compacts submanifolds $Z_1,\dots, Z_k$, of $\mX$, all lying in the central fiber such that    

   \begin{equation}\label{eqDFintro}
 \DF(\mX,\metric) = \sum_{i=1}^k \int_{Z_i} \frac{\mA_{i,\Omega}^n \wedge \mB_{i,\Omega}}{\euler(N_{Z_i}^\mX)}
  \end{equation} where $\mA_{i,\Omega}, \mB_{i,\Omega}\in \oplus_m \Gamma(\bigwedge^m T^*{Z_i})$ are closed forms explicitly given in Proposition~\ref{theoLOCDFinvariant} and $\euler(N_{Z_i}^\mX)$ is the equivariant Euler class of the normal bundle of $Z$ in $\mX$. 

Using this expression, we show that $\DF(\mX,\metric)$ coincides with the classical Futaki invariant of the induced $\C^*$--action on the central fiber when the central fiber is smooth or admits orbifold singularities. 

\begin{theorem}\label{theoINTRO1} Let $(\mX^{n+1},[\metric])$ be a regular (compact) test configuration over $(X^n,[\omega])$ with $\pi: \mX \ra \bP^1$ and the $\bC^*$--action $\nu: \bC^* \hookrightarrow \mbox{Aut}(\mX)$. Let $V=\nu_*(\partial \theta)$ be the vector field induced by the underlying $S^1$--action and $\mom :\mX\ra \bR$ be any Hamiltonian function for $V$.
Assume that the central fiber $X_0:=\pi^{-1}(0)$ inherits a K\"ahler orbifold structure from the inclusion $\iota_{0}: X_0\hookrightarrow \mX^{n+1}$. Then    
     \begin{equation}\label{eqDFinvLOCintro}\begin{split}
  \frac{\DF(\mX,[\metric])}{n!}  = -\pi \emph{Fut}_{(X_0,[\metric_0])}(J_0V)
  \end{split}
  \end{equation} where $\metric_0=\iota_{X_0}^*\metric$ is the pull-back on $X_0$. 
\end{theorem}

Test configurations with smooth central fiber but with $\pi$ not being a submersion are known to exist see~\cite{ADVLN,Tian}. Examples with $\pi$ being a submersion also fulfill the hypothesis of the Theorem, such examples include the so-called "product tests configuration" see \S\ref{subsecPRODUCT}.

Donaldson gave in \cite{don:scalar} a very simple proof of this fact in the polarized case when the central fiber is smooth using the (equivariant) Hirzebruch--Riemann--Roch formula and the definition of $\DF$ as a coefficient of the expansion of the normalized weight of the induced action on $H^0(\mL^k)$. Our approach is differential-geometric and does not distinguish the polarized and transcendental cases but doesn't provide a more direct proof. A feature of our method is that it provides a formulation of the Donaldson--Futaki invariant as the evaluation of some classes on the (compact smooth) manifolds lying in the possibly singular central fiber, see Proposition~\ref{theoLOCDFinvariant}. Using this and the conclusion of~\S\ref{ssectEqClassDF}, as well as \cite[Proposition 6]{EG} it is clear that $\DF$ is an equivariant class on the central fiber, see~\S\ref{subsectionSINGcf}. The hypothesis of Theorem~\ref{theoINTRO1} is used to make sense of that equivariant class as the classical Futaki invariant.      

Another asset of our proof is that it should be valid when replacing the K\"ahler structures involved by almost-K\"ahler ones which might be useful in the attemp to extend the Yau--Tian--Donaldson conjecture to this setting, see eg~\cite{KL}.

 We prove Theorem~\ref{theoINTRO1} using the Equivariant Localization Formula of Atiyah--Bott\cite{AtiyahBott} and Berline--Vergne~\cite{BerlineVergne}, recalled in Section~\ref{secEQUIVcohomology}. The Equivariant Localization Formula has been used before in K\"ahler geometry \cite{DingTian, Fu2, MaschlerThesis, Tian}. In particular, Tian used it successfully to exhibit examples of Fano K\"ahler manifolds admitting no compatible and smooth K\"ahler--Einstein metric \cite{TIANbookalg}. Also Wang used it, on $\bP^1$, to give the intersection formulation of the Donadson--Futaki invariant \cite{wang}. This latter technic has been used recently successfully by Inoue in \cite{inoue} where he highlights also the equivariant interection formula \eqref{eqEQUIVfut} of the invariant as we do here. \\

 In Section~\ref{s:DFnormalCone}, we use our formula over fixed point set to compute the Donaldson--Futaki invariant of tests configuration obtained by deformation to the normal cone of any subvariety $Y\subset X$ and relate it to the Futaki invariant of the exceptional divisor.

{\bf Aknowledgement} I have been very talkative about this modest project. I thank the following people for their interest in this work: V. Apostolov, P. Carrillo-Rouse, R. Dervan, P. Gauduchon, H. Guenancia, L. Manivel, G. Maschler and D. Witt Nystr\"om. \\

\section{The equivariant cohomology and localization formula} \label{secEQUIVcohomology}

 We briefly introduce the tools of equivariant cohomology we will use later. We refer to the books \cite{BerlineGetzlerVergne,GSbook} for more details and proof. We express the results for compact torus acting on smooth compact (symplectic or complex) manifold but the theory is developed in more complicated case. 
 
\subsection{Equivariant cohomology}
Let $M$ be a smooth manifold with an effective and smooth action of a compact torus $T$, that is $\nu: T \hookrightarrow \mbox{diffeo}(M)$. We will mostly work with the induced infinitesimal action of $\kt:=\mbox{Lie } T$. For $a\in \kt$, we denote the induced vector field on $M$ by $V_a=\phi_*(a)$, equivalently $$V_a(p) =\frac{d}{dt}_{t=0} \nu(\exp(t a))(p).$$      

An equivariant form is a polynomial map $\biho : \kt \ra \Omega^*(M)^T$ where $\Omega^*(M)^T$ is the graded complex of $T$-invariant forms. We denote $\Omega^*_{T}(M)$ the set of equivariant forms over $M$. There is an appropriate notion of degree for these forms so that the equivariant differential $d_\kt: \Omega^*_{T}(M)\longrightarrow \Omega^*_{T}(M)$ which is defined by $$a\; \mapsto \; (d_{\kt} \biho)_a := (d-{ }_{V_a}\intprod)\biho_a$$ increases the degree of $1$. Here ${ }_{V_a}\intprod \biho$ is the contraction of $\biho$ by $V_a$.  An equivariant form $\biho\in \Omega^*_{T}(M)$ is said {\it equivariantly closed} if $d_{\kt} \biho\equiv 0$.

\begin{example}\label{exampleSYMP}
 Consider $(M^{2n}, \omega)$ a symplectic compact manifold admitting a Hamiltonian action of a torus $T$, $\nu: T\hookrightarrow \mbox{Ham}(M,\omega)$ with momentum map $\mom : M\ra \kt^*$, that is $$-d\langle\mom, a\rangle =\omega(V_a,\cdot)$$ for any $a\in\kt$. The form $a\mapsto \omega -\langle\mom,a\rangle$ is equivariantly closed, as well as, $(\omega -\langle\mom,\cdot\rangle)^k$ and, thus, $e^{\omega -\langle\mom,\cdot\rangle}:=\sum_{k=0}^{+\infty} \frac{(\omega -\langle\mom,\cdot\rangle)^k}{k!}$.  \\
 \end{example}
    
    It is known that $d_\kt^2=0$ and the equivariant cohomology can be defined as the cohomology of the associated chain complex: $$H_T^*(M):= \frac{\ker d_\kt }{\mbox{im }d_\kt}$$ see~\cite{BerlineGetzlerVergne}. There is another definition of equivariant cohomology introduced by Atiyah and Bott~\cite{AtiyahBott} and valid for non connected compact Lie group $G$, which is the cohomology of the quotient $M_G:= EG\times_G M$ where $EG\ra BG$ is the universal principal $G$--bundle. In the case of connected compact Lie group both definitions coincide.

\subsection{Localization formula}

Let $\bfZ := \mbox{Fix}_M T$ be the fixed points set of $T$ in $M$. It is known, see eg. \cite{BerlineGetzlerVergne,GSbook}, that $\bfZ$ consists in a disjoint union of smooth submanifolds $\bfZ=\sqcup Z$ of even codimension. Given a connected component, say $Z \subset\bfZ$ of codimension $2\ell_Z$, the normal bundle $\bm : N_Z^M \longrightarrow Z$ bears a complex structure induced by the action \cite[\S 8.5]{GSbook} and at any given point $z\in Z$, $(N_Z^M)_z$ splits into a sum of (complex) lines $(N_Z^M)_z=\oplus_{j} N^Z_{z,j}$ according to the action of $T$, acting on $N^Z_{z,j}$ with weight $\we^Z_j\in\kt^*$, here $j$ runs from $1$ to $\ell_Z$.  

The {\it equivariant Localization Formula} as we use many times in this note is the following. 
\begin{theorem}\label{extDHteo}[Atiyah--Bott \cite{AtiyahBott}, Berline--Vergne \cite{BerlineVergne}] Let $\biho \in \Omega^*_{T}(M)$ be an equivariently closed form. For any generic $a\in \kt$ we have \begin{equation}\label{formDHext}
 \int_M \biho_a = \sum_Z \int_Z \frac{\iota_Z^*\biho_a }{\euler(N^Z_a)}.
\end{equation} $\euler(N^Z)^{-1}$ is the inverse of the equivariant Euler class of $N_Z^M$ in $H_{S^1}^*(Z)$. 
\end{theorem} 

Here, a point of $\kt$ is said generic if the zero set of $V_a$ is the fixed points set of $T$. \\ 

Recall that the equivariant Euler (respectively Todd, Thom, Chern...) class of a rank $r$ $G$--equivariant bundle $E\ra M$ is the Euler (respectively Todd, Thom, Chern...) class of the corresponding rank $r$ bundle $E_G\ra M_G$. When the splitting into line bundles is global $N_Z^M=\oplus_{j} N^Z_{j}$ the equivariant Euler class of $N_Z^M$ is just \begin{equation}\label{eq:equivEULERsplit} a\mapsto  \euler(N^Z_a)= \frac{1}{(-2\pi)^{\ell_Z}}\prod_{j=1}^{\ell_Z} (2\pi c_1(N^Z_j) - \langle\we^Z_j,a \rangle).\end{equation}

\begin{rem} We use a different convention than most authors (see eg. \cite{BerlineGetzlerVergne}) in incorporating $(-2\pi)^{\ell_Z}$ in the notation of $\euler(N^Z_a)$.  
\end{rem}

\begin{example}\label{exampleP1} Consider the Riemann sphere $\bP^1=\bC\sqcup \{\infty\}$ with the complex structure induced by $\bC \subset \bP^1$, coordinates $z=x+iy= re^{i\theta}$ with $0\leq r\leq +\infty$, metric $$g_{FS}= \frac{4}{(1+r^2)^2} (dr^2+r^2d\theta^2)$$ and symplectic form $\omega_{FS} = \frac{4}{(1+r^2)^2} rdr\wedge d\theta$. The total volume is then $4\pi$, the metric is K\"ahler-Einstein with Einstein constant equals $1$ and scalar curvature $2$. The {\it standard action} of $S^1$ is the one induced by the vector field $\partial_\theta:= x\partial_y- y\partial_y$, that is $$(e^{i \theta},z) \mapsto e^{i \theta}z,$$ a momentum map is $\mom_{FS} = (r^2-1)/(r^2+1)$ so that $\int_\bP^1 \mom_{FS}\omega_{FS}=0$ and the weight of the action on $T_0\bP^1$ is $1$ and the action on $T_\infty\bP^1$ is $-1$. \\  
\end{example}

\section{Localizing the Futaki invariant}\label{secFUTAKI}
We recall how one can use the equivariant Localization Formula~\eqref{formDHext} to give an alternative formula for the Futaki invariant \cite{MaschlerThesis}. This was done in various setting in~\cite{ DingTian, Fu2, Tian}.

Consider a K\"ahler compact manifold $(M^{2n},g,\omega,J)$ with K\"ahler class ${[\omega]} \in H^2_{dR}(M,\R)$. The Futaki invariant, see \cite{futaki}, is a character on the space of holomorphic vector fields and can be alternatively defined on the space of real holomorphic vector fields $\mF_{[\omega]} : \mathfrak{h} \ra \bR$ as $$\mF_{[\omega]}(V)= \int_M (V.\phi)\frac{\omega^n}{n!}$$ where $\rho^g= \rho_H^g+dd^c\phi$ is the Ricci form of $(M^{2n},g,\omega,J)$ and $\rho_H^g \in 2\pi c_1(M)$ is the harmonic representative. We have the decomposition, see eg \cite[Lemma 2.1.1]{PGnote}, $$V^\flat=\xi_H+df^V +d^ch^V$$ where $\xi_H$ is a harmonic $1$--form and $f^V,h^V\in C^\infty(M)$. The Futaki invariant is alternatively given by $$\mF_{[\omega]}(V)= \int_M (f^V-\overline{f^V})s_g\frac{\omega^n}{n!}$$ where $s_g$ is the scalar curvature and $\overline{f^V} =\int_M f^V\omega^n/\int_M \omega^n$, see \cite[\S 4.11]{PGnote}. In particular, for any $a\in \kt$ we have
\begin{equation}
-\mF_{[\omega]}(JV_a) = \int_M (\langle \mom,a\rangle- \overline{\langle \mom,a\rangle}) s_g \frac{\omega^n}{n!} = 2 \int_M (\langle \mom,a\rangle- \overline{\langle \mom,a\rangle})  \frac{\rho^g\wedge\omega^{n-1}}{(n-1)!}. 
\end{equation} Equivalently, 

\begin{equation}
-\mF_{[\omega]}(JV_a) = 2 \int_M \langle \mom,a\rangle  \frac{(\rho^g- \cst\omega)\wedge\omega^{n-1}}{(n-1)!} 
\end{equation} for $\cst=\cst_{[\omega]}:= \int_M \rho^g\wedge\omega^{n-1}/ \int_M \omega^{n}$. Recall, see eg \cite[Lemma 1.23.4]{PGnote}, that the Bochner formula gives in the K\"ahler setting $${}_{V_a}\intprod \rho = -\frac{1}{2}d\Delta^g \mom_a$$ where we put $\mom_a:=\langle \mom,a\rangle$. We get two equivariantly closed forms $$\A_{\omega,a}:= a\mapsto (\omega- \mom_a) \mbox{ and } a\mapsto \B_{\omega,a} =\left( \frac{n\cst}{n+1}(\omega- \mom_a)- \left(\rho^g-\frac{1}{2}\Delta^g \mom_a\right)\right)  $$ such that   
\begin{equation}\label{eqEQUIVfut}
\mF_{[\omega]}(JV_a) = \frac{2}{n!} \, [\A_{\omega,a}]^n\cup [\B_{\omega,a}] \, ([M])
\end{equation} in equivariant cohomology. 

\begin{rem}\label{remarkINDEPclass} The equivariant class of $(\omega - \mom_a)$ only depends on the de Rham class of $[\omega]$. Indeed, if $\omega-\omega' =dd^c\phi$ then $(\omega' - (\mom_a- d^c\phi(V_a)))$ is equivariently closed and the difference is equivariantly exact $$(\omega - \mom_a)- (\omega' - (\mom_a- d^c\phi(V_a))) = d_\kt (d^c\phi)_a.$$ Using the same argument $[\B_{\omega,a}]= [\B_{\omega',a}]$ in equivariant cohomology. 
\end{rem}

Recall that the momentum map $\mom$ is a constant on any connected component $Z\subset \mbox{Fix}_T(M)$. Moreover, for any $a\in \kt$ the set of values $\{\mom_a(Z)\}_Z$ depends on the choice of the momentum map and on the representative in the de Rham class $[\omega]$ only up to an overall constant. We put $$\valueMOM^Z_{{[\omega]}}(a) := \mom_a(Z)$$ and recall that for any $z\in Z$, 
\begin{equation}\label{eq:LAPweightsFIXEDpoints}(\Delta^{g}\langle\mom, b\rangle)_z = -2\langle \sum_{j=1}^{n-n_Z} \we^Z_j, b\rangle\end{equation} for a proof of this see eg.~\cite{BHL_DH}. Therefore, using Theorem~\ref{extDHteo} and \eqref{eqEQUIVfut} for $f=\frac{1}{2}\Delta^{g}\mom_a$, we get that 

\begin{equation}\label{eq:FUTinvLOCALIZED}
\begin{split}
\mF_{[\omega]}(JV_a) &= \sum_Z  \frac{ n\cst_{[\omega]}({[\omega]} - \valueMOM^Z_{{[\omega]}}(a))^{n+1}}{(n+1)!\euler(N^Z_a)} ([Z])\\
-2\sum_Z &  \frac{(2\pi c_1(M) + \sum_{i=1}^{n-n_Z}\langle \we_i, a\rangle)\cup ({[\omega]}- \valueMOM^Z_{{[\omega]}}(a))^n}{n!\euler(N^Z_a)} ([Z])
\end{split}
\end{equation}

\begin{rem} The first line of \eqref{eq:FUTinvLOCALIZED} vanishes if we take the normalization $\int_{M}\mom \omega^n =0$ which turns $\mF_{[\omega]}(JV_a)$ into a homogenous polynomial of order $n$ in the variable ${[\omega]}$. 
\end{rem}

Formulation~\eqref{eq:FUTinvLOCALIZED} has been used in~\cite{DingTian} to extend the notion of Futaki invariant to singular K\"ahler manifolds, in particular orbifolds. Recall that K\"ahler geometry can be extended to orbifolds since the analysis is essentially the same as explained in~\cite{DingTian} and many other works. Be aware, as highlighted in the first section of \cite{RT}, there is a difference between complex orbifold singularity and a (K\"ahler) metric orbifold singularity. The analysis is similar to smooth manifold setting in the second case which is the case we consider here (this is also the case considered in the aforementioned works). Indeed, an orbifold metric is a metric which pulls back to a smooth invariant metric on the uniformizing charts, then it makes sense to define the curvature, Ricci curvature and so on.

Using the localization principle for orbifolds, see~\cite{Meinrenken}, we get
\begin{equation}\label{eq:FUTinvLOCALIZEDorbifold}
\begin{split}
\mF_{[\omega]}(JV_a) &= \sum_Z  \frac{ n\cst({[\omega]} - \valueMOM^Z_{{[\omega]}}(a))^{n+1}}{d_Z(n+1)!\euler(N^Z_a)} ([Z])\\
-2\sum_Z & \frac{(2\pi c_1(M) + \sum_{i=1}^{n-n_Z}\langle \we_i, a\rangle)\cup ({[\omega]}- \valueMOM^Z_{{[\omega]}}(a))^n}{d_Z n!\euler(N^Z_a)} ([Z])
\end{split}
\end{equation} where $d_Z$ is the number of elements in the isotropy group of generic points of $Z$ and $N^Z$ is the equivariant normal bundle of $Z$ in the uniformizing charts.   

\section{Regular compact test configurations}\label{subsectTCDF0}

\subsection{Compact regular test configurations and the Donaldson--Futaki invariant}\label{subsectTCDF}

Following~\cite{DervanRoss,zakarias}, we call a (regular) {\it test configuration} for a K\"ahler manifold $(X,\omega)$ a following set of data 

\begin{enumerate}
 \item A smooth compact K\"ahler manifold $(\mX,\metric)$;
 \item $\bC^*$--action $\nu: \bC^* \hookrightarrow \mbox{Aut}(\mX)$ such that $[\nu(t)^*\metric]=[\metric]$;
 \item a $\bC^*$--equivariant surjective map $\pi : \mX\longrightarrow \bP^1$ for the standard action on $\bP^1:= \bC \cup \{+\infty\}$;
 \item a $\bC^*$--equivariant biholomorphism $$\biho: \mX^*:= \mX\backslash \pi^{-1}(0) \stackrel{\sim}{\longrightarrow} X\times \bP^1\backslash \{0\}$$ for the trivial action on $X$ times the restriction of the standard action on $\bP^1$ and $\mbox{pr}_2(\biho(x))= \pi(x)$.   
\end{enumerate}
Moreover, for $t\in \bP^1$, we denote $X_t :=\pi^{-1}(t)$ and the inclusion $\iota_t : X_t\hookrightarrow \mX$. We have, for $t\neq 0$, $X_t\stackrel{\biho_t}{\simeq} X$, that is $\biho_t = \mbox{pr}_1\circ \biho\circ \iota_t$. Finally we denote $\metric_t =\iota_t^*\metric$ and assume that
\begin{equation}\label{eqCDTcohomoTC}
 [\metric_t]= [\biho_t^*\omega].
\end{equation}

In this context, the {\it Donaldson--Futaki} invariant of $(\mX,\metric)$, when $\dim_\bC X=n$, can be defined to be: 
\begin{equation}\label{eqDFinv}
\DF(\mX,\metric) = \frac{ n}{n+1} \cst_{[ \omega]} \int_\mX \metric^{n+1} - \int_{\mX} (\rho^\metric -\pi^*\rho^{FS})\wedge \metric^n  
\end{equation}
where $\cst_{[\omega]} = \int_X\rho^\omega\wedge\omega^{n-1} / \int_X\omega^{n}.$

The quantity \eqref{eqDFinv} coincides with the invariant introduced by Donaldson in \cite{don:scalar} when the test configuration is the compactification of a polarized test configuration as shown in Odaka~\cite{odaka} and Wang~\cite{wang}, see also \cite[p.315]{don:scalar}.  

\begin{rem} In the K\"ahler setting~\cite{DervanRoss,zakarias} the test configurations over $(X,[\omega])$ considered are more generally (compact) K\"ahler analytical spaces. Then the Donadson-Futaki invariant as \eqref{eqDFinv} is defined for a smooth resolution over it. It is proved then that the definition is independant of the resolution and that to test $K$-semistability (i.e that $\DF(\mX,[\metric]) \geq 0$ for all test configurations over $(X,[\omega])$) and even uniform K-stability (see eg.\cite{DervanRoss}) it is enough to consider smooth test configurations, see \cite[Proposition 2.23]{DervanRoss}.\end{rem}

\begin{rem} The definition of a test configuration and of its Donaldson--Futaki invariant is independant of the representative $\metric$ chosen in $[\metric]\in H^2_{dR}(\mX)$ and is usually denoted $\DF(\mX,[\metric])$. For the techniques of equivariant cohomology we apply in this paper it is convenient to pick a $S^1$--invariant representative in $[\metric]$ where we understand $S^1 \hookrightarrow \bC^*$ in the standard way. By abuse of notation we will often avoid the brackets. 
\end{rem}

\subsection{The Donaldson--Futaki invariant as an equivariant classes product}\label{ssectEqClassDF}
{\it From now on, without loss of generality, we pick a $S^1$-invariant K\"ahler metric $\metric$ on $\mX$.}

We denote $V:= \nu_*(\partial_\theta)$ the real holomorphic vector field on $\mX$ induced from the generator of $S^1$ via the action. By assumption this is also a Killing vector field with zeros and thus it is a Hamiltonian vector field on $(\mX,\metric)$. We pick a Hamiltonian function $\mom : \mX\ra \bR$ that is $$-d\mom=\metric(V,\cdot).$$   
 We get again two equivariantly closed forms $\mA_{\metric}:= (\metric- \mom)$ and $$ \mB_{\metric} = \frac{n\cst}{n+1}(\metric- \mom)- \left(\rho^\metric-\frac{1}{2}\Delta^\metric \mom\right) + (\pi^*\omega_{FS}- \pi^*\mom_{FS}) $$ where $\mom_{FS}$ is a Hamiltonian for the standard $S^1$ action on $(\bP^1,\omega_{FS})$ see Example~\ref{exampleP1}. Since the integration picks up only the $2(n+1)$--degree terms we have 
 \begin{equation}\label{eqEQUIVfut}
\DF(\mX,\metric) =  \, [\mA_{\metric}]^n\cup [\mB_{\metric}] \, ([\mX])
\end{equation} in equivariant cohomology. As in Remark~\ref{remarkINDEPclass}, the equivariant classes $[\mA_{\metric}]$, $[\mB_{\metric}] \in H_{S^1}^2(\mX)$ are independant of the chosen representative in $[\metric] \in H_{dR}^2(\mX,\bR)$.\\

Many times in this paper we will use the following facts and notation: since $\pi: \mX \ra \bP^1$ is equivariant the fixed points set of the $S^1$ action on $\mX$, $\mbox{Fix}_{S^1}\mX$, lies in $X_0\sqcup X_\infty$.  Following a classical argument of Riemannian geometry,  $\mbox{Fix}_{S^1}\mX$ is a disjoint union of smooth submanifolds, $\mbox{Fix}_{S^1}\mX=X_\infty\sqcup \left(\sqcup_l Z_l\right)$. One of the component is $X_\infty$ because of the condition (4) defining a test configuration. The remaining components lie in $X_0$. We denote then $$\mbox{Fix}_{S^1}\mX = (\sqcup Z)  \sqcup X_\infty$$ where $Z$ denotes a generic component of $\mbox{Fix}_{S^1}\mX$ in $X_0$.

\section{Localizing the Donaldson--Futaki invariant} \label{sectLOCDF}

Applying Theorem~\ref{extDHteo} we will prove the next proposition where the following notation is used: given any subset $S\subset \mX$ we denote the inclusion $\iota_S:S\ra \mX$ and, for any form/function/tensor.. $\alpha$ on $\mX$ the pull-back to $S$ is denoted $\alpha_S:= \iota^*_S\alpha$.

\begin{proposition}\label{theoLOCDFinvariant} Let $(\mX,[\metric],\nu,\pi, \biho)$  be a regular (compact) test configuration over $(X^n,[\omega])$. We pick $\metric\in [\metric]$ a $S^1$--invariant K\"ahler metric. Let $V=\nu_*(\partial \theta)$ be the vector field induced by the underlying $S^1$--action, $\mom :\mX\ra \bR$ be any Hamiltonian function for $V$. The Donaldson--Futaki invariant of  $(\mX,[\metric],\nu,\pi, \biho)$ is  
       \begin{equation}\label{eqDFinvLOCintrolap}\begin{split}
 \frac{\DF(\mX,\metric)}{n!}  = &\sum_Z \int_Z \frac{n\cst_{[\omega]}(\metric_Z - \mom_Z)^{n+1}}{(n+1)! \euler(N_Z^\mX)(V)}\\
  &- \sum_Z\int_Z\frac{(\rho^\metric_Z +\sum_{i=0}^{n-n_Z}\langle\we_i,\partial \theta\rangle )\wedge(\metric_Z - \mom_Z)^{n}}{n! \euler(N_Z^\mX)(V)}\\ 
  &+ \sum_Z\int_Z\frac{(\metric_Z - \mom_Z)^{n}}{n! \euler(N_Z^\mX)(V)}.
  \end{split}
  \end{equation} where $Z$ runs into the set of connected components of the fixed points set of the $S^1$--action lying in the central fiber $X_0=\pi^{-1}(0)$. The equivariant vector bundle $N_Z^\mX$ over $Z$ is the normal bundle of $Z$ in $\mX$ and $\we_0^{Z},\dots,\we_{n-n_Z}^Z\in(\mbox{Lie } S^1)^*$ are the weights (with multiplicity) of the induced action of $S^1$ on $N_Z^\mX$.
 \end{proposition}

 \begin{rem} A more concise way to express the last claim using the notation introduced in\S\ref{ssectEqClassDF} is the following 
 $$\DF(\mX,\metric) = \sum_Z \int_Z  \frac{\iota_Z^*\mA_{\metric}^n\wedge \iota_Z^*\mB_{\metric,o}}{\euler(N_Z^\mX)(V)}$$
 where $\mB_{\metric,o} = \frac{n\cst_{[\omega]}}{n+1}\mA_{\metric} - (\rho^\metric +\sum_{i=0}^{n-n_Z}\langle\we_i,\partial \theta\rangle) + 1$ and $\iota_Z: Z\hookrightarrow \mX$ is the inclusion.  
 \end{rem}

\begin{proof}
 
We use the notation of the last paragraph and the observation and notation lay down in the last subsection. In particular, the normal bundle of $X_{\infty}$ in $\mX$, denoted $N^\infty$, is trivial and the weight of the induced action is $-1$. Consequently the Euler equivariant form (with the normalisation used here~\eqref{eq:equivEULERsplit}) is $$\euler(N^{X_\infty}_V)= - (2\pi c_1(N^{X_\infty}) - \langle \we^{\infty},\partial \theta  \rangle)/2\pi =-1/2\pi.$$

Now, as explained in~\S\ref{ssectEqClassDF}, see formula~\eqref{eqEQUIVfut} we can write the Donaldson--Futaki invariant as the integral of an equivariantly closed form. We apply Theorem \ref{extDHteo}, using $\euler(N^{X_\infty}_V)=1$ and writting $f=\frac{1}{2}\Delta^\metric \mom$, we get  
 \begin{equation*}\begin{split}
 \frac{\DF(\mX,\metric)}{n!} & = \sum_Z  \int_Z \frac{n\cst_{[\omega]}(\metric_Z - \mom_Z)^{n+1}}{(n+1)! \euler(N_Z^\mX)_V} - \frac{(\rho^\metric_Z -f_Z)\wedge(\metric_Z - \mom_Z)^{n}}{n! \euler(N_Z^\mX)_V} \\
 & \qquad \qquad+  \int_Z\frac{\iota_Z^*(\pi^*\rho^{FS} - \pi^*\mom_{FS})\wedge(\metric_Z - \mom_Z)^{n}}{n! \euler(N_Z^\mX)_V}\\
 & -2\pi \int_{X_\infty}  \frac{n\cst_{[\omega]}(\metric_\infty -\mom_\infty)^{n+1}}{(n+1)!} - \frac{(\iota_\infty^*\rho^\metric -f_{X_\infty} )\wedge(\metric_\infty - \mom_\infty)^{n}}{n!}\\
  & \qquad \qquad -2\pi \int_{X_\infty}\frac{\iota_\infty^*(\pi^*\rho^{FS} - \pi^*\mom_{FS})\wedge(\metric_\infty - \mom_\infty)^{n}}{n!}
 \end{split}
 \end{equation*} where $\cst_{[\omega]}= \cst(X,[\omega])$. We will see that the integral over $X_\infty$ vanishes up to one term. Note that $\iota_\infty^*\pi^*\rho^{FS}=0$ and on $X_\infty$, $\iota_\infty^*(\Delta^{\metric}\mom) = -2\langle \we^{\infty}, \partial \theta \rangle =2$ and $\mom_{FS}(\infty)=1$ see Example~\ref{exampleP1}. The splitting $\iota_\infty^*T\mX = TX_\infty \oplus\pi^*T_\infty\bP^1$ implies that $[\iota_\infty^*\rho^\metric]=2\pi c_ 1(X_\infty)= [\rho^{\metric_\infty}]$ in cohomology. The integral over $X_\infty$ is then 
  \begin{equation}\label{eqintXinfty}\begin{split}
 &\int_{X_\infty}  \frac{-n\cst_{[\omega]} \mom_\infty\metric_\infty^n}{n!} - \frac{(\rho^{\metric_\infty} - f_\infty)\wedge(\metric_\infty - \mom_\infty)^{n}}{n!} - \frac{(\metric_\infty - \mom_\infty)^{n}}{n!}\\
 &=\int_{X_\infty}  \frac{-n\cst_{[\omega]} \mom_\infty\metric_\infty^n}{n!} - \frac{\rho^{\metric_\infty}\wedge(\metric_\infty - \mom_\infty)^{n}}{n!} +(f_\infty -1)\frac{(\metric_\infty - \mom_\infty)^{n}}{n!}\\
  &= \int_{X_\infty} \frac{-n\cst_{[\omega]} \mom_\infty\metric_\infty^n}{n!} + \frac{n \mom_\infty \rho^{\metric_\infty}\wedge \metric_\infty^{n-1}}{n!} +(f_\infty -1)\frac{(\metric_\infty - \mom_\infty)^{n}}{n!}.
 \end{split}
 \end{equation} The first two terms cancel each others because $\cst_{[\omega]}= \cst(X,[\omega])=\cst(X_\infty,[\metric_\infty])$. 
 
  The remaining components $Z$ of the fixed points set lies in $X_0=\pi^{-1}(0)$ so $\iota_Z^*\pi^*\rho^{FS} =0$ and $\iota_Z^*\pi^*\mom_{FS} =-1$. Moreover since $f_{\infty}=1$ and $f_Z = -\sum_{i=0}^{n-n_Z}\langle\we_i,\partial \theta\rangle$, we end up with \eqref{eqDFinvLOCintrolap}. \end{proof}

  \subsubsection{Singular central fiber}\label{subsectionSINGcf}
  Proposition~\ref{theoLOCDFinvariant} gives the $\DF$ invariant as an intersection number on compact submanifolds lying in the possible singular central fiber. In the next two subsections we show that using back the equivariant localization formula~\eqref{formDHext} the right hand side of \eqref{eqDFinvLOCintrolap} is nothing but the Futaki invariant of (the induced $S^1$ action on) the central fiber under some strong conditions ensuring this last invariant to be defined as it is classically~\cite{futaki} and recalled in \S\ref{secFUTAKI}. As one can check in the subsequent section, the issue is to relate the equivariant Euler forms $\euler(N^\mX_Z)$ and $\euler(N^{X_0}_Z)$. To get a valid statement when $X_0$ is singular one needs to refer to equivariant cohomology for singular variety. This theory has been developed in details by Edidin and Graham in the context of algebraic spaces and schemes, in terms of cycles and equivariant Chow groups~\cite{EG_intersectionTHEORY}. They also proved a localization formula~\cite{EG} which fits well in our setting, namely the singular variety has to be embedded in a smooth one. This embedding allows them to use the normal bundle as we get in the right hand side of \eqref{eqDFinvLOCintrolap}. Therefore, starting with a test configuration bearing an algebraic structure (eg. polarized scheme see \cite{don:scalar}) and up to translating statements about equivariant forms into statement involving equivariant cycles, we can apply directly Edidin-Graham result~\cite[Proposition 6]{EG} to state that in this case $$\DF(\mX,[\metric]) = \iota_0^*[\mA_{\metric}]^n\cup \iota_0^*[\mB_{\metric}]([X_0]).$$  Indeed, as an equivariant ($S^1$--stable) subvariety $X_0 \subset \mX$ the forms $\iota_0^*\mA_{\metric}$ and $\iota_0^*\mB_{\metric}$ are also equivariantly closed forms on $X_0$ (seen as a K\"ahler space). Moreover, in anycase (even when $X_0$ is singular), Proposition~\ref{theoLOCDFinvariant} tells us that the contribution to $\DF$ coming from $X_\infty$ is trivial.

\subsection{Smooth central fiber}

\begin{theorem}\label{theoSMOOTHcentralFiber} Let $(\mX,[\metric],\nu,\pi)$ be a regular (compact and smooth) test configuration and let $X_0$ be the central fiber on which the $S^1$ action induced by $\nu$ gives the vector field $V=\nu_*\partial_\theta \in\Gamma(TX_0)$. Assume that the central fiber $X_0$ is a smooth submanifold of $\mX$, then \begin{equation}\label{FUT=DFsmooth}\frac{\DF(\mX,[\metric])}{n!} =- \pi \mF_{[\metric_0]}(J_0V)\end{equation} where $(\metric_0,J_0)$ denotes the K\"ahler structure of $X_0$ induced from $\mX$.       
\end{theorem}

\begin{proof} We use the notation introduced in \S\ref{secEQUIVcohomology} and denote $(\metric,J, \cG)$ a $S^1$--invariant K\"ahler structure on $\mX$. We will apply Proposition~\ref{theoLOCDFinvariant}, more precisely formula \eqref{eqDFinvLOCintrolap}, with a Hamiltonian function $\mom :\mX\ra \bR$ for $V=\nu_*(\partial\theta)$ normalized by \begin{equation}\label{normMOM0}\int_{X_0} \mom_0\, \metric_0^n =0 \end{equation} where, as usual, we denote $\iota_{X_0}^*\mom=\mom_0$.
   
Fix a connected component $Z \subset \mbox{Fix}_{\mX} S^1$ lying in $X_0$. The sequence of $S^1$--equivariant embeddings $Z\subset X_0 \subset \mX$ gives a short exact sequence of vector bundles over $Z$ 
$$0\longrightarrow N^{X_0}_Z \longrightarrow  N^\mX_{Z} \longrightarrow  N^\mX_{X_0} \longrightarrow 0,$$ which, in turn, gives $$\euler(N^\mX_{Z}) = \euler(N^{X_0}_Z )\cup \iota_Z^*\euler(N^\mX_{X_0}).$$ 
We have a decomposition into $\cG$--orthogonal $S^1$--equivariant bundles
\begin{equation}\label{decompPOINTsmooth}
\iota^*_0T\mX =  TX_0 \oplus E_0\end{equation} where $E_0\simeq N^\mX_{X_0}$ equivariantly. We put $E_0^Z:=\iota_Z^*E_0$. Observe that this line bundle corresponds to one summand $(E_0^Z)_z$ in the equivariant decomposition of $T_z \mX$ for any $z\in Z$. We will show that the weight associated to this summand is $1$. 

%

Pick $\mathcal{A} \subset E_0^Z$ a $S^1$--invariant tubular neighbourhood of the zero section in $E_0^Z\ra Z$. Since $X_0$ is smooth and $Z$ compact we may assume that the restriction of the Riemannian exponential map $\exp : \mathcal{A} \ra \mX$ is injective on $\mathcal{A}$. Moreover, each fiber $\mathcal{A}_z := \mathcal{A}\cap T_z\mX$ is sent to a submanifold tranversal to $X_0$. Taking a smaller tubular neighbourhood if necessary, the map $\pi \circ \exp$ is $S^1$--equivariant and its image contains some $t\in\bP^1\backslash\{0,\infty\}$. From this we deduce easily that $\langle\we_0,\partial_\theta\rangle =\pm 1$.

The pull back $(\pi\circ\exp_z)^*(\mom_{FS})$ is a smooth $S^1$--invariant function along the fiber $\mathcal{A}_z$. Thus, by a classical result of Schwarz~\cite{Schwarz} it is a smooth function of the fiberwise squared norm induced by $G$ on $E_0^Z$, which coincides, up to the weight and an additive constant with $\exp^*\mom$, see~\cite{GS} and \cite[\S IV.1.b]{Audin}. Now, since $G(\nabla^G \mom, \nabla^G \pi^*(\mom_{FS})) = \pi^*g_{FS}(V,V)>0$ on $\pi^{-1}(\bP^1\backslash\{0,\infty\})$ we have that $(\pi\circ\exp_z)^*(\mom_{FS})$ is an increasing function of $\exp^*\mom$ along each fiber of $E_0^Z$. Therefore the weight of the $S^1$-action on $E_0^Z$ is positive and thus $\langle\we_0,\partial_\theta\rangle = 1$.

A key observation is that, because $\pi$ is equivariant, $\bP^1\backslash \{0,\infty\}$ are all regular values of $\pi$ and then the normal bundles $N_t:=N_{X_t}^{\mX}$ of $X_t$ in $\mX$ are all trivial. This implies that on the regular part of $X_0$ (which is the whole of $X_0$ under the hypothesis here) the normal bundle $E_0\simeq N^\mX_{X_0} \ra X_0$ is, at least topologically, trivial. Indeed on the regular part any open set $U\subset \mX$, $U\cap (X_0)_{reg}$ can be approximated smoothly by $U\cap (X_t)$. Another way to be convinced of this is to build a section of $N_{(X_0)_{reg}}^\mX$ as follow. For a point $x\in (X_0)_{reg}$ an equivariant neighbourhood $U_x\subset \mX$ of the orbit $S^1\cdot x =\{\nu(\lambda)x\,|\, \lambda \in S^1\}$ can be described via the symplectic slice Theorem, see~\cite{GS,LT}. So that, as Hamiltonian $S^1$--space, $U_x$ is isomorphic to an open equivariant neighbourhood of the zero section in $S^1 \times_{S^1_x} (\bR \times V_x)$ or $T_xS^1$ depending if $S^1_x$ is finite or not and where $V_x = (T(S^1\cdot x))^{\perp_\Omega} /(T(S^1\cdot x)^{\perp_\Omega}\cap TS^1\cdot x)$. Therefore the level sets of $exp_x^*|\pi|^2$ (an increasing function of $exp_x^*\mom_{FS}$) and of $exp_x^*\mom$ can be compared in a neighbourhood of $0\in T_xN_{(X_0)_{reg}}$ as we did above, to conclude that for $t\in \bP^1$ with $|t|$ small enough there is only one single point in $\exp_x(U_x\cap (N_{(X_0)_{reg}})_x) \cap X_t$. With this we can build the desirable section $\sigma_t \in \Gamma((X_0)_{reg},N_{(X_0)_{reg}}^\mX)$.

Putting all that together we have:     
\begin{itemize}
 \item[a) ] $\iota_{Z}^*[\rho^\metric] =  \iota_{Z}^*[\rho^{\metric_0}]\in H_{dR}^2(Z)$;
 \item[b) ] $\langle\we^Z_0,\partial \theta\rangle = 1$; 
 \item[c) ] $\euler(N_Z^\mX) = - \euler(N_Z^{X_0})/2\pi$.
\end{itemize} 

Condition c) allows us to consider the terms of \eqref{eqDFinvLOCintrolap} as integrals of equivariantly closed forms on $X_0$ using the Localization Formula~\eqref{extDHteo}. Indeed the first line of \eqref{eqDFinvLOCintrolap} is the integral of the equivalently closed form $2\pi (\metric_0 - \mom_0)^{n+1}$ on $X_0$. Then the first line of \eqref{eqDFinvLOCintrolap} is canceled by the normalization \eqref{normMOM0}. 

Thanks to condition b) the middle term in the Formula~\eqref{eqDFinvLOCintrolap} is canceled with the last term (the third line). We are left with \begin{equation*}
 \frac{\DF(\mX,\metric)}{n!}  =-\sum_Z\int_Z\frac{(\rho^\metric_Z +\sum_{i=1}^{n-n_Z}\langle\we^Z_i,\partial \theta\rangle )\wedge(\metric_Z - \mom_Z)^{n}}{n!(- \euler(N_Z^{X_0}))(V)}.
  \end{equation*} Condition a) together with the fact that the integrant above are constant or closed forms on $Z$ gives that
   \begin{equation*}
 \frac{\DF(\mX,\metric)}{n!}  =\sum_Z\int_Z\frac{(\iota_Z^*\rho^{\metric_0} +\sum_{i=1}^{n-n_Z}\langle\we^Z_i,\partial \theta\rangle )\wedge(\metric_Z - \mom_Z)^{n}}{n! \euler(N_Z^{X_0})(V)}.
  \end{equation*} Now, we have $(\Delta^{\metric_0}\mom_0)_z  = -2\sum_{i=1}^{n-n_Z}\langle\we_i,\partial \theta\rangle $ for any $z\in Z$ and thus \begin{equation*}
 \frac{\DF(\mX,\metric)}{n!}  =\sum_Z\int_Z\frac{(\iota_Z^*\rho^{\metric_0} -\frac{1}{2}\iota_Z^*(\Delta^{\metric_0}\mom_0) )\wedge(\metric_Z - \mom_Z)^{n}}{n! \euler(N_Z^{X_0})(V)}
  \end{equation*} where $\rho^{\metric_0} -\frac{1}{2}(\Delta^{\metric_0}\mom_0)$ is an equivariantly closed form on $X_0$. This allows us to use the Localization formula again and we get \begin{equation*}
 \frac{\DF(\mX,\metric)}{n!}  =-2\pi \int_{X_0}\frac{(\rho^{\metric_0} -\frac{1}{2} \Delta^{\metric_0}\mom_0)\wedge(\metric_0 - \mom_0)^{n}}{n!}= -\pi \mF_{[\metric_0]}(J_0V)
  \end{equation*} thanks to the Formula~\eqref{eqEQUIVfut} and using \eqref{normMOM0} again. \end{proof}

\subsubsection{Product test configurations} \label{subsecPRODUCT} When a K\"ahler manifold $(X,\omega)$ admits a holomorphic action $\nu: \bC^*\hookrightarrow \mbox{Aut}(X)$ restricting to a Hamiltonian isometric $S^1$--action $\nu: S^1\hookrightarrow \mbox{Aut}(X,\omega)$, then one can build a test configuration $(\mX_\nu,\metric_\nu)$ where the total space $\mX_\nu$ is obtained as the {\it clutching construction} from $X$ and the action $\nu$. One way to define $(\mX_\nu,\metric_\nu )$ is via a K\"ahler reduction, see \cite{BHJ} for the projective case. 

First consider the K\"ahler manifold $(X\times \bC^2,\omega + \omega_{std})$ with a $\bC^*$ action whose induced $S^1$ action admits the moment map $$\tilde{\mom}(x,z,w) = \mom(x) - \frac{1}{2}(|z|^2+|w|^2)$$ where $\mom: X\ra \bR$ is the moment map picked for the action on $(X,\omega)$. Note that any $c<\min \mom$ is a regular value of $\tilde{\mom}$ and the $S^1$ action on $\tilde{\mom}^{-1}(c)$ is free. Therefore $\mX_{\nu}=\tilde{\mom}^{-1}(c)/S^1$ is a manifold which inherits of a K\"ahler structure $(J_\nu,\metric_{\nu,c})$, see eg. Futaki\cite{kirwan,Fu3}, so that the map $\pi([x,z,w]) = [z:w]$ from $\mX_{\nu}$ to $\bP^1$ is holomorphic and equivariant. The underlying complex manifold is independant of the choice of $c$ chosen into a connected set of regular values of $\tilde{\mom}$ thanks to \cite{kirwan}. The K\"ahler class $[\metric_{\nu,c}]$ depends on the value $c \in (-\infty, \min \mom)$ but in any case $\iota_t^* [\metric_{\nu,c}]=[\omega]$ as one can show following the ideas in~\cite{McDuffTolman}. Then $(\mX_\nu,\metric_\nu,\pi,\nu)$ satisfies the conditions of a test configuration enumerated in \S\ref{subsectTCDF}. Such a test configuration satisfies the hypothesis of Theorem~\ref{theoSMOOTHcentralFiber} and thus we have $\DF(\mX_\nu,\metric_\nu) =- n!\pi \mF_{\metric}(J\nu_*(\partial\theta)).$
%
%

\subsection{Orbifold central fiber}\label{sectOrbiCF}

\subsubsection{The equivariant geometry of the orbifold central fiber}\label{sectOrbiCF1}
Assume now that the central fiber $X_0$ of the (smooth, compact) test configuration $(\mX,\metric,\pi,\nu)$ is not smooth but inherits from $(\mX,\metric)$ of a K\"ahler orbifold structure. That means that for each singular point $x\in X_0$ there is an open set $U\subset \mX$ such that $\metric_0 = \iota_{X_0}^*\metric$ defines a distance on the set $U\cap X_0$, which is isometric with an orbifold metric on some $\tilde{U}/\Gamma$ where $\Gamma \subset \mbox{GL}(\bC^n)$ is finite and $\tilde{U}$ is a $\Gamma$ invariant open subset of $\bC^n$ see eg \cite{LT,RT}. In particular, each connected component of the singular set has codimension at least $2$ in $X_0$ and $X_0$ is irreductible as a complex space.  

Moreover, the Futaki invariant of the induced $\bC^*$ action on $(X_0,\metric_0 = \iota_{X_0}^*\metric)$ makes sense, see \cite{DingTian} and \S\ref{secFUTAKI}, and coincides with \eqref{eq:FUTinvLOCALIZEDorbifold}.

We first give some consequences of the assumptions above. Each connected component $Z$ of the $S^1$--fixed point set is a smooth K\"ahler submanifold of $(\mX,\metric)$ as recalled in Section~\ref{secEQUIVcohomology}.  A connected (smooth) compact metric manifold cannot be isometric to a metric orbifold with more than one type of isotropy group, see~\cite[Lemma 2.1]{LangeOrbifoldMetric}. Therefore, $\Gamma_z\simeq \Gamma_Z$ for all $z\in Z$. Moreover, $Z$ is compact as the connected component of the zero set of a smooth vector field on a compact manifold $\mX$, and thus $Z$ is covered by a finite number of open sets of the form $U\cap Z$ for open set $U\subset \mX$ such that $U\cap X_0$ satisfies the orbifold condition recalled above and that $\quot_U^{-1}(z) \in \widetilde{U}$ is a single point for all $z\in Z$. The subset $\quot_U^{-1}(Z\cap U)$ is fixed by $\Gamma_Z$, smooth in $\tilde{U}$, and mapped homeomorphically onto $Z\cap U$ via $\quot_U$. The collection of the uniformizing charts $\widetilde{U}$ together with the transition maps condition define a differential manifold $\widetilde{U}_Z$ with an action of the abstract group $\Gamma_Z$ so that $Z\subset \widetilde{U}_Z$ is a smooth submanifold fixed by $\Gamma_Z$ and there is a homeomorphism $\quot_Z:\widetilde{U}_Z/ \Gamma_Z \ra \cup U\cap X_0$. That is, there is a neighbourhood of $Z$ inside $X_0$ which is the global quotient of a manifold by a finite group $\Gamma_Z$ acting freely on $\widetilde{U}_Z \backslash Z$ and fixing $Z$. We sum up the discussion in the following claim.

\begin{lemma}
 Let $(\mX,\metric,\nu,\pi)$ be a regular (compact and smooth) test configuration and let $(X_0,\metric_0)$ be the central fiber. Assume that $X_0$ inherits of a structure of metric orbifold from $(\mX,\metric)$ then each connected component of the singular locus either coincide with a component of the fixed points set or is disjoint of the fixed points sets. 
\end{lemma}

The normal bundle of $Z$ in $\widetilde{U}_Z$ inherits of $\Gamma_Z$ action, $\tilde{\bm} : N_Z^{\widetilde{U}} \ra Z$ and the normal orbibundle of $Z$ in $X_0$ is the quotient $N_Z^{\widetilde{U}}/\Gamma_Z$. We denote $\bm : N_Z^{X_0} \ra Z$, the orbibundle map.

These open sets $U\subset \mX$, covering $Z$, can be chosen equivariant with respect to the $S^1$--action since $Z$ is fixed by $S^1$ and thus $U\cap X_0$ are also equivariant. Then, following~\cite{LT}, there is a group extension $\Gamma_Z \hookrightarrow H_U \stackrel{q}{\twoheadrightarrow} S^1$ and an action $\tilde{\nu} : H_U\hookrightarrow \mbox{Isom} (\tilde{U},\tilde{g})$ covering the initial one via $\quot_U$. Since $\Gamma_z$ is constant on $Z$, the extension is also constant, say $H_U\simeq H_Z$, and we get an isometric (and biholomorphic) action of $H_Z$ defined globally on the manifold $\widetilde{U}_Z$. There are two possibilities: either $H_Z$ is a product of $S^1$ with a finite group ($\Gamma_Z$ or a quotient of it) or $H_Z\simeq S^1$ if $H_Z$ is connected. The example of Ding--Tian~\cite{DingTian} falls into the first category and any toric examples would give the second~\cite{LT}. In either case the connected component to the identity is isomorphic to $S^1$ and the normal bundle $N^{\widetilde{U}_Z}_Z$ is then a $S^1$--equivariant bundle over $Z$ with weights $\w_1,\dots, \w_{\ell_Z}\in (\mbox{Lie} S^1)^*\simeq \bR$ (the weights might not lie in the (same) weights lattice of our original representation of $S^1$ but it is not a problem in terms of localization formula~\cite{Meinrenken}).

\subsubsection{The Donaldson-Fuaki invariant and the Futaki invariant of the orbifold central fiber}\label{sectOrbiCF2}

\begin{theorem}\label{theoORBIFOLDcentralFiber} Let $(\mX,\metric,\nu,\pi)$ be a regular (compact and smooth) test configuration and let $(X_0,\metric_0)$ be the central fiber on which the $S^1$ action induced by $\nu$ gives the vector field $V=\nu_*\partial_\theta$. Assume that $X_0$ inherits of a structure of metric orbifold from $(\mX,\metric)$. Then \begin{equation}\label{FUT=DFsmooth}\frac{\DF(\mX,\metric)}{n!} =- \pi \mF_{[\metric_0]}(J_0V)\end{equation} where $(\metric_0,J_0)$ denotes the K\"ahler structure of $X_0$ induced from $\mX$.       
\end{theorem}


 Ding and Tian produced examples of such test configurations in the polarized setting~\cite{DingTian}. 

\begin{proof} Observe from the proof of Theorem~\ref{theoSMOOTHcentralFiber} that it is suffisant to show that these three following conditions hold on each connected component $Z$ of the fixed point set:
\begin{itemize}
 \item[a) ] $\iota_{Z}^*[\rho^\metric] =  \iota_{Z}^*[\rho^{\metric_0}]\in H_{dR}^2(Z)$;
 \item[b) ] $\sum_{i=0}^{\ell_Z}\langle\we^Z_i,\partial \theta\rangle = 1+ \sum_{i=1}^{\ell_Z}\langle\w^Z_i,\partial \theta\rangle$; 
 \item[c) ] $\euler(N_Z^\mX) = -d_Z \euler(N_Z^{X_0})/2\pi$;
\end{itemize} 
here $d_Z$ is the order of the isotropy group of a generic point in $Z$, $N_Z^{X_0}$ is the $S^1$--equivariant bundle over $Z$ constructed above and weights $\w_1,\dots \w_{\ell_Z}$.

We fix a connected component $Z\subset X_0\cap(\mbox{Fix} S^1)$ for the rest of the discussion and denote $\ell_Z$ its complex codimension in $\mX$. 

A key observation is that, because $\pi$ is equivariant, $\bP^1\backslash \{0,\infty\}$ are all regular values of $\pi$. Therefore the normal bundles $N_t:=N_{X_t}^{\mX}$ of $X_t$ in $\mX$ are all trivial. This implies that $\iota_{X_t}^*[\rho^\metric] = \iota_{X_t}^*[\rho^{\metric_t}]$ by the Adjunction formula. Then $\rho^{\metric_t}\longrightarrow \rho^{\metric_0}$ since it holds on the regular part and do not blow-up along singular locus thanks to the orbifold condition. Hence, we get condition a). 

Pick a point $z\in Z$ and consider the decomposition 
\begin{equation}\label{decompPOINT}
T_z\mX =  T_zZ \oplus (\oplus_{j=0}^{\ell_Z} E_{j,z})\end{equation} induced by the isometric and Hamiltonian $S^1$-action. This decomposition is $\cG$--orthogonal, $\metric$--orthogonal and $J$--invariant as explained for example in \cite[\S IV.1.b]{Audin}. Also the induced action of $S^1$ is Hamiltonian and isometric and provides a vector field $W_z \in \Gamma(TT_z\mX)$ with Hamiltonian function $H\in C^\infty(T_z\mX)$ whose Laplacian at $0$ is $$-\frac{1}{2} \sum_{j=0}^{\ell_Z} \we_j$$ which is equivalent to Formula~\eqref{eq:LAPweightsFIXEDpoints}. Let consider the Riemannian exponential map $\exp_z: T_z\mX \longrightarrow \mX$ associated to $\cG$. This map is $S^1$ equivariant and is a diffeomorphism onto its image when restricted to an open $S^1$-equivariant neighbourhood, say $\mB$, of $0\in T_z\mX$. Pick a point $v\in \mB$ such that $\exp_z v \notin X_0$, such point exists by an argument involving the dimension of $X_0\cap U \backslash (Z\cap U)$. Since $\exp_z$ and $\pi$ are $S^1$-equivariant and $\pi(\exp_z v)\neq 0$, the stabilizor of $\exp_z v$ in $S^1$ must be trivial. Comparing the gradients of $H=\exp_z^*\mom$ and $\exp_z^*|\pi|^2$ one can argue, as in the proof of Theorem~\ref{theoSMOOTHcentralFiber}, that the punctured disk $\bD^*= \{\lambda \in \bC^*\,|\, |\lambda|\leq 1\}$ is sent into $\mB$ via the map $$\lambda \mapsto d_z\nu(\lambda)v  \in T_z\mX$$ and $d_z\nu(\lambda)v\longrightarrow 0$ when $\lambda \ra 0 \in \bC$.              


Taking $\widetilde{U}_Z$ the uniformizing chart of $Z$ with the notation in\S\ref{sectOrbiCF1}, the $S^1$-equivariant decomposition of $$T_z\widetilde{U}=T_zZ \oplus (\oplus_{j=1}^{\ell_Z} \widetilde{E}_{j,z})$$ also behaves nicely with respect to the K\"ahler structure $(\tilde{\metric_z},\tilde{G_z},\tilde{J})$ induced by the (orbifold) map $\iota_0\circ\quot_Z : \widetilde{U} \ra U\subset  \mX$. We take an $S^1$-equivariant neighbourhood $\tilde{\mB}$ of $0\in T_z\widetilde{U}$ such that the Riemannian exponential map $\widetilde{\exp_z} : \tilde{\mB} \ra \widetilde{U}$ is a diffeomorphism over its image and consider the map $$\psi_z(w,\lambda) = \exp_z^{-1}\circ\quot_U\circ\widetilde{\exp_z}(w) + d_z\nu(\lambda)v \in T_z\mX.$$ Up to taking a smaller $S^1$-equivariant set $\tilde{\mB}$ this map is well defined and continous on $\tilde{\mB} \times \bD^*$ and extend smoothly on $\tilde{\mB} \times \bD$. First, observe that $\psi_z$ is a local diffeomorphism when restricted to $(\tilde{\mB}\backslash (T_zZ\cap \tilde{\mB}) \times \bD^*$ therefore the image is open. Also $\psi_z: \tilde{\mB} \times \bD \ra T_z\mX$ is $S^1$-equivariant for the action $\tilde{\nu}$ times the standard action on $\bD$. This is the crucial feature of this map. We can consider the $S^1$--invariant metric structure induced by $\psi_z$ on $(\tilde{\mB}\backslash (T_zZ\cap \tilde{\mB})) \times \bD^*$, this coincides with the previous one on $\tilde{\mB}$ times the standard structure on $\bD^* \subset \bC$ (we don't mean here that the Riemannian exponential is an isometry but we use it as a $S^1$--equivariant diffeomorphism to identify subsets of the bundles and the fact that $\exp_z\circ d_z\quot_U $ should coindice with $\widetilde{\exp_z}$ by definition of the metric structure on $\widetilde{U}$). Therefore $\psi_z^*H$ is also the Hamiltonian function for the Killing and Hamiltonian vector field induced by the $S^1$-action on $T_z\tilde{U}\times \bC$ and $$\Delta^{\tilde{G_z} + \omega_{std}}(\psi_z^*H)_{(0,0)} = (\Delta^{G_z }H)_0$$ which means, using Formula~\eqref{eq:LAPweightsFIXEDpoints}, that we get condition b) above.                                 

It remains to prove that condition c) holds with the hypothesis of the Theorem. One idea would be to argue that even if the normal and tangent bundles of $X_0$ are not well defined as bundles, they exist as sheaves and therefore the Adjunction formula together with a) gives $c_1(N_{X_0}^{\mX})=0$ and the weight of the $S^1$ action on $N_{X_0}^{\mX}$ is $1$. Then the extension of the equivariant cohomology theory to sheaves should ensure that the sequence of $S^1$-equivariant maps $$N_Z^{X_0} \stackrel{d\quot}{\longrightarrow} N_Z^{\mX}\longrightarrow N_{X_0}^{\mX}\longrightarrow 0$$ implies that $\euler(N_Z^\mX) = d_Z\euler(N_{X_0}^{\mX}) \euler(N_Z^{X_0})$ and therefore we get condition c). Another, more pedestrian, approach is to define a $S^1$--equivariant map, say $\Psi$, from a tubular neighbourhood, $\widetilde{N}\times \underline{\bD}$ of the zero section of $\tilde{p} : N_Z^{X_0}\times \underline{\bC} \ra Z$ to $\mN$, a tubular neighbourhood of the zero section of $p:N_Z^{\mX}\ra Z$. To be useful here this map should pull back the (compactly supported) equivariant Thom form of $N_Z^{\mX}$ to $d_Z$ times the one of $N_Z^{X_0}\times \underline{\bC}$. Indeed, recall that an equivariant Thom form $\tau(E)$ of a rank $2r$ equivariant bundle $E\ra Z$ is any equivariantly closed form that integrates to $(2\pi)^r$ over each fiber, see eg. \cite{BerlineGetzlerVergne,GSbook}. There are compactly supported version of these forms. An important feature of a Thom form (which all lie in the same cohomology class of course) is that it pulls back to $Z$, seen as the zero section, to give a representative of the equivariant Euler class of the bundle. 
Therefore it is sufficiant to construct a map $\Psi$ with domain and target described above such that $\frac{1}{d_Z}\Psi^*(\tau(N_Z^{\mX}))$ satisfies the properties of a Thom form and $\Psi(Z) = Z$. 
We now construct this map. First consider the bundle $L=N_{(X_0)_{reg}}^\mX\longrightarrow U\backslash Z$ over $U\backslash Z$ which we identify with a subbundle of $T(U\backslash Z)$ using the metric. This bundle is trivial as explained before. We pick a $S^1$--invariant nowhere vanishing section $\sigma \in \Gamma(U\backslash Z, L)$ such that $\sigma_x \ra 0$ when $x\ra Z$. The set $U$ is a tubular neighbourhood of $Z$ in $\mX$ and $L$ is a trivial bundle on $U\backslash Z$ so it is clear that such a section exists. We can take $\sigma$ lying into the subset where the exponential map is injective. Then we define 
$$\Psi(w,\lambda) := \exp_{\tilde{p}(w)}^{-1}\circ\quot_U\circ\widetilde{\exp_{\tilde{p}(w)}}(w) + d_{\tilde{p}(w)}\nu(\lambda)\exp_{\tilde{p}(w)}^{-1}\exp_x(\sigma_x)$$
where $x=\quot_U\circ\widetilde{\exp_{\tilde{p}(w)}}(w)\in X_0$

This map is continuous, $S^1$ equivariant, local diffeomorphism away from $Z$ (seen as a zero section of $N_Z^{X_0}$) and is surjective. Combined with the fact that $\int_{U_0} \alpha= \frac{1}{d_Z}\int_{\quot_U^{-1}(U_0)}\quot_U^*\alpha$ for any form $\alpha \in \Omega^*(U_0)$ on any open set $U_0\subset U\cap X_0$ we check easily that $\frac{1}{d_Z}\Psi^*(\tau(N_Z^{\mX}))$ is a Thom form on $N_Z^{X_0}\times \underline{\bC}$ and that $\Psi(Z)=Z$. \end{proof}
%
%

\section{Deformation to the normal cone}\label{s:DFnormalCone}
\subsection{Deformation to the normal cone as a test configuration} 

\subsubsection{Construction and fixed points set}
Let $Y^m\subset X^n$ be a smooth compact connected subvariety of complex codimension $k=n-m$. We consider the blow-up of $X\times \bP^1$ along $Y\times\{0\}$ and call it $\mX_Y$ or $\mX$. As explained in \cite{RTslope}, see also~\cite{DervanRoss} for the adaptation to compact non-projective test configurations, this gives a test configuration over $(X,[\omega])$ for a suitable K\"ahler class $[\mA]$ on $\mX$. We now recall various aspects of this test configuration, called the {\it deformation to the normal cone}. 

We denote the blow-down map $\b: \mX\longrightarrow X\times \bP^1$ and $\pi=\mbox{pr}_2\circ \b : \mX\ra \bP^1$ the test configuration surjective map. The $\bC^*$--action on $\mX$ is the only one that covers the standard one on $\bP^1$ via $\pi$, we call it $\phi: \bC^*\hookrightarrow \mbox{Aut}(\mX)$. The $S^1$--action given via $\phi$ gives Killing vector field $\phi_*(\partial\theta)=V$ on $(\mX,\metric)$ and we pick a Hamiltonian $\mom:\mX \ra \bR$. 

The central fiber of $\pi$ is \begin{equation}\label{centralFIBERdefNormCONE}\pi^{-1}(0) = b^{-1}(X\backslash Y \times \{0\})\sqcup \mE\end{equation} where $\mE=b^{-1}(Y\times \{0\})\subset \mX$ is the exceptional divisor. Recall that $\mE \simeq \bP(N_{Y\times\{0\}}^\mX)=\bP(N_{Y}^X\oplus \pi^*T_0\bP^1)$ where $N_Y^X=TX/TY$ (respectively $N_Y^{\mX}=T\mX/TY$) is the normal bundle of $Y$ in $X$ (respectively in $\mX$). Moreover, the closure of $b^{-1}(X\backslash Y \times \{0\})$ in $\mX$ will be denoted $\bu$, we have $\bu \simeq \mbox{Bl}_Y(X)$ the blow-up of $X$ along $Y$. We denote the blow-down map $\check{\b} : \bu \ra X$ and $\check{\b}^{-1}(Y)=\excep\simeq \bP(N_{Y}^X)$ the exceptional divisor. Of course, $\mE \cap \bu =\excep$ in $\mX$.


In the regular test configuration $(\mX,[\metric_s], \pi,\phi)$, the $S^1$--invariant metric $\metric_s$ can be chosen to satisfy the requirements of a test configuration, see \cite{DervanRoss,RTslope} as well as \begin{equation} 
\metric_{s,\mE}:= \iota_\mE^*\metric_s  =b^*\iota_Y^*\omega -s \delta_{\mE} \end{equation} where $\delta_{\mE}$ is the curvature of an $S^1$--invariant hermitian bundle metric on $\mO_\mE(-1) \ra \mE$ and $s >0$ is some positive constant see eg.~\cite{voisin}.\\

\subsubsection{Equivariant Euler classes of the normal bundles} \label{subsecNORMALdeformCONE}

The fixed points set of the action of $\phi: \bC^*\hookrightarrow \mbox{Aut}(\mX)$ is $$X_\infty\sqcup Z_0\sqcup \bu$$ where $Z_0=\bP(0\oplus T_0\bP^1)\simeq Y$ and $\bu\simeq \mbox{Bl}_Y(X)$ are defined above. 

\begin{notation}\label{notationBUNDLELblowup}
It is well-known, see eg.~\cite[Lemmes 3.25, 3.26]{voisin}, that on any blow-up there exists a holomorphic line bundle, unique up to isomorphism, trivial outside the  exceptional divisor and whose restriction to the exceptional divisor is the tautological line bundle which agrees with the normal bundle. That is there is a line bundle $\mL\ra \mX$ over the blow-up $\mX= \mbox{Bl}_{Y\times\{0\}}(X\times \bP^1)$ satisfying $\mL_{|_{\mX\backslash \mE}} \simeq \underline{\bC}$ and $\mL_{|_{\mE}} \simeq \mO_{\mE}(-1)\simeq N^\mX_\mE$. It is straightforward to check that the restriction of that bundle $L := \mL_{|_{\bu}}$ satisfies $L_{|_{\bu\backslash \excep}} \simeq \underline{\bC}$ and $L_{|_{\excep}} \simeq \mO_{\excep}(-1)\simeq  N_{\excep}^{\bu}$.   
\end{notation}

Observe that $Z_0=\bP(0\oplus T_0\bP^1)$ is a smooth submanifold of $\mE$ which is itself a smooth divisor in $\mX$, closed for the action of $S^1$ induced by $\phi$. Therefore the normal bundle of $Z_0=\bP(0\oplus T_0\bP^1)$ in $\mX$ splits equivariantly into \begin{equation}\label{NORM_Z_0splits}N_{Z_0}^\mX = \iota_{Z_0}^*N_{\mE}^\mX \oplus N_{Z_0}^\mE\end{equation} where $N_{\mE}^\mX\simeq \mO_{\mE}(-1)$. By definition the action $\phi$ on $\mX$ gives the multiplication on the second factor of the bundle $N_{Y}^X\oplus \pi^{*}T_0\bP^1$ with weight $1$ (see Example~\ref{exampleP1}). Moreover, $\iota_{Z_0}^*\mO_{\mE}(-1)$ is a trivial bundle over $Z_0$. Therefore the equivariant Euler forms satisfy 
\begin{equation}\label{eqEULERformZ0} \euler(N_{Z_0}^\mX)(V) = - \euler(N_{Z_0}^\mE)(V)/2\pi.\end{equation}

The set $\excep=\bP(N_Y^X\oplus 0)$ is a smooth divisor in $\mE$ and a connected component of the fixed points set of the $S^1$--action induced by $\phi$ on $\mE$. One can check, by restricting to the fibers $\excep_y \simeq \bP^{k-1}\subset\bP^{k} \simeq \mE_y$, that the normal bundle $N_{\excep}^\mE$ is isomorphic to $\imath^*_{\excep}\mO_\mE(1)=\mO_{\excep}(1)$. 



The set $\excep$ is also a smooth submanifold of $\bu$ which is itself a (smooth) divisor in $\mX$, as a connected component of the fixed point set of holomorphic action. We get an equivariant splitting $N_\excep^\mX \simeq N_\excep^\bu \oplus \imath^*_\excep N_\bu^\mX$. The induced $S^1$--action on $N_\excep^\bu$ is trivial because $\bu$ is fixed (so the weight is $0$) and on $\imath^*_\excep N_\bu^\mX$ the weight has to be $1$ because this is what it is along $\bu\backslash \excep$ where $\pi$ is a submersion. There is an equivariant inclusion of bundles $N_{\excep}^\mE \hookrightarrow N_\excep^\mX$ and since the set of weights of the equivariant decomposition of $N_\excep^\mX$ is well defined up to re-ordering we have that the weight of the $S^1$--action is $1$.        

Therefore, denoting $\delta  = 2\pi c_1(\mL)$ the class mentioned above, so $$\delta_\mE:=\imath^*_\mE \delta = 2\pi c_1(\mO_{\mE}(-1))\, \mbox{ and }\,\delta_\excep:=\imath^*_\excep \delta = 2\pi c_1(\mO_{\excep}(-1)),$$ we deduce that
\begin{equation}\label{eqEulerclassesDEFORM}
\euler(N_\excep^\mE) = (\delta_\excep +1)/2\pi\, \mbox{ and }\, \euler(N_\bu^\mX) = (\delta_\bu +1)/2\pi. 
\end{equation}

%
%

%
%

\subsection{The Donaldson--Futaki invariant of the deformation to the normal cone}

We use Formula~\eqref{eqDFinvLOCintrolap} to compute the Donaldson--Futaki invariant of the deformation to the normal cone of $Y\subset X$ as described above.

 \begin{theorem}\label{theoDEFORM_NC}  Let $(\mX,[\metric],\phi,\pi)$ be a regular (compact and smooth) test configuration over $(X,[\omega])$ obtained by the blow-up of $X\times \bP^1$ along $Y\times \{0\}$, a connected compact smooth subvariety of $X$. Then for any function $\mom :\mX \ra \bR$, Hamiltonian for the $S^1$--action induced by $\phi$ and normalized by the condition 
 \begin{equation}\label{formLOCmEnormMOMdeformCONE}
  \int_{\mE}\iota_\mE^*(\mom \metric^{n})=0.
 \end{equation} we have  
   \begin{equation}\label{eqDFdeformCONE}\begin{split}
  \frac{\DF(\mX,\metric)}{n!} = & -2\pi \mF_{\metric_{\mE}}(JV) \\
  & + 2\pi (-n\mom_\bu) (\cst_{[\omega]} - \cst_{[\metric_\bu]}) \int_\bu\metric_\bu^n + 4\pi \int_{\bu}   \frac{(\metric_\bu-\mu_\bu)^n\wedge \delta^2_\bu}{\delta_\bu +1} 
 \end{split}
 \end{equation} where $\cst_{[\metric_\bu]} = \int_\bu\rho^{\metric_\bu}\wedge \metric_\bu^{n-1} / \int_\bu\metric_\bu^n$ and $\delta_\bu \in 2\pi\iota^*_\bu c_1(\mL)$ (see Notation~\ref{notationBUNDLELblowup}).
 \end{theorem}
 
 The analysis of the weights of the action made in \S\ref{subsecNORMALdeformCONE} tells us that $\bu$ is the global minimum of $\mu$ and then \eqref{formLOCmEnormMOMdeformCONE} implies that $\mu_\bu<0$.
 
 The case of most interest is certainly when $Y$ is a divisor in $X$. In the smooth polarized setting, it is the only case that has to be taken into account according to \cite{RTslope}. Note that, in this case, $\bu\simeq X$ and the middle term of \eqref{eqDFdeformCONE} vanishes up to a normalization and the meaningful term seems to be the Futaki invariant of $\bP(N_Y^X\oplus \bC)$.
 
 \begin{rem}
  Formulas similar to \eqref{eqDFdeformCONE} appear in \cite{RTslope,RTmumford} and are used successfully to give many examples of unstable polarized K\"ahler manifolds and interesting relation between scalar curvature and Seshadri constant.    
 \end{rem}

%

\subsubsection{Proof of Theorem~\ref{theoDEFORM_NC}} We pick as before a $S^1$--invariant metric $\metric$ on $\mX$ and a function $\mom :\mX \ra \bR$, Hamiltonian for the $S^1$--action induced by $\phi$ and normalized by the condition~\eqref{formLOCmEnormMOMdeformCONE}. Using Formula~\eqref{eqDFinvLOCintrolap} the Donaldson--Futaki invariant of the test configuration $(\mX,\metric,\phi,\pi)$ is the sum of $3$ integrals over each connected component of the fixed points set lying in the central fiber. In the case of the deformation to the normal cone there are two such components $Z_0$ and $\bu$, described above \S\ref{subsecNORMALdeformCONE}. We study first the integrals over $Z_0$, which is, using the notation of Section \ref{sectLOCDF}, 
$$\int_{Z_0}  \frac{\iota_{Z_0}^*\mA_{\metric}^n\wedge \iota_{Z_0}^*\mB_{\metric,o}}{\euler(N_{Z_0}^\mX)(\partial \theta)}$$ where
$\mB_{\metric,o} = \frac{n\cst_{[\omega]}}{n+1}\mA_{\metric} - (\rho^\metric +\sum_{i=0}^{\ell_{Z_0}}\langle\we_i,\partial \theta\rangle) + 1$ and $\ell_{Z_0}$ is the complex codimension of $Z_0$ in $\mE$. According to the discussion of the last section, one of the weight of the induced action on the normal bundle $N_{Z_0}^\mX$ is $1$, say $\langle \we_0^{Z_0},\partial_{\theta}\rangle =1$, corresponding to the summand $\iota_{Z_0}^*N_\mE^\mX$ in \eqref{NORM_Z_0splits}. It gives $\mB_{\metric,o} = \frac{n\cst_{[\omega]}}{n+1}\mA_{\metric} - (\rho^\metric +\sum_{i=1}^{\ell_{Z_0}}\langle\we_i,\partial \theta\rangle)$.

Using \eqref{eqEULERformZ0} we get 
\begin{equation}\label{eqdeformDF_Z0_1}\int_{Z_0}  \frac{\iota_{Z_0}^*\mA_{\metric}^n\wedge \iota_{Z_0}^*\mB_{\metric,o}}{\euler(N_{Z_0}^\mX)(\partial \theta)}= -2\pi\int_{Z_0}  \frac{\iota_{Z_0}^*\mA_{\metric}^n\wedge \iota_{Z_0}^*\mB_{\metric,o}}{\euler(N_{Z_0}^\mE)(\partial \theta)}. \end{equation} Each form appearing in the last integrant is closed so we can use the facts $[\rho^{\metric}] = [\rho^{\metric_{\mE}}] + \classO_{\mE}$ and $\iota_{Z_0}^*\classO_{\mE}=0$ to replace $\rho^{\metric}$ by $\rho^{\metric_{\mE}}$ in the last integral. Now, $\rho^{\metric_\mE} -\frac{1}{2}\Delta^{\metric_\mE}\mom_{\mE}$ is equivariantly closed on $\mE$ (with respect to the action induced by $\phi$) and $\iota_{Z_0}^*\frac{1}{2}\Delta^{\metric_\mE}\mom_{\mE} = -\sum_{i=1}^{\ell_{Z_0}}\langle \we_i^{Z_0},\partial_{\theta}\rangle$. So, using~\eqref{formDHext}, we get that \eqref{eqdeformDF_Z0_1} becomes 
\begin{equation}
 -2\pi\int_{Z_0}  \frac{\iota_{Z_0}^*\mA_{\metric}^n\wedge \iota_{Z_0}^*\mB_{\metric,o}}{\euler(N_{Z_0}^\mE)(\partial \theta)}= -2\pi\int_{\mE} \iota_{\mE}^*\mA_{\metric}^n\wedge \iota_{\mE}^*\mB_{\metric_{\mE},o}  + 2\pi \int_{E}  \frac{\iota_{E}^*\mA_{\metric}^n\wedge \iota_{E}^*\mB_{\metric_{\mE},o}}{\euler(N_{E}^\mE)(\partial \theta)} 
\end{equation}

The first term of $\int_{\mE} \iota_{\mE}^*\mA_{\metric}^n\wedge \iota_{\mE}^*\mB_{\metric,o}$ is cancelled thanks to the normalisation~\eqref{formLOCmEnormMOMdeformCONE} so it gives $\int_{\mE} \mA_{\metric_\mE}^n\wedge (\rho^{\metric_\mE} -\frac{1}{2}\Delta^{\metric_\mE}\mom_{\mE}) $ which is the Futaki invariant of the action induced by $\phi$ on $\mE$. We end up with 

\begin{equation}\label{eqdeformDF_Z0_2}
 -2\pi\int_{Z_0}  \frac{\iota_{Z_0}^*\mA_{\metric}^n\wedge \iota_{Z_0}^*\mB_{\metric,o}}{\euler(N_{Z_0}^\mE)(\partial \theta)}= -2\pi \mF_{\metric_{\mE}}(JV)  + 2\pi \int_{E}  \frac{\iota_{E}^*\mA_{\metric}^n\wedge \iota_{E}^*\mB_{\metric_{\mE},o}}{\euler(N_{E}^\mE)(\partial \theta)} 
\end{equation} Using \eqref{eqEulerclassesDEFORM}, the second term of the right hand side is 
\begin{equation}\label{eqdeformDF_Z0termE_1}
 (2\pi)^2 \int_{E}  \frac{\mA_{\metric_E}^n\wedge \iota_{E}^*\mB_{\metric_{\mE},o}}{\delta_E +1}
\end{equation}  where $\iota_{E}^*\mB_{\metric_{\mE},o} = \frac{n\cst_{[\omega]}}{n+1}\mA_{\metric_E} - (\rho^{\metric_\mE} +1)$.
 As before, we use that each form appearing in the last integrant is closed to replace $\iota^*[\rho^{\metric_\mE}]$ by any cohomologous form, namely by any representant of $\iota_E^*[\rho^{\metric_{\mX}}]  - 2\pi{c_1}(N_\mE^\mX).$ Moreover, by unicity of the equivariant decomposition of the normal bundle of $E$ in $\mX$ we know that $N_E^\bu \simeq \iota^*_EN_\mE^\mX \simeq \mO_E(-1)$ so that $2\pi c_1(N_E^\bu)=\delta_E$ and using \cite[p.263 (11.3.1)]{voisin} the integral over $E$ becomes 
 
 \begin{equation}\label{eqdeformDF_Z0termE_2}
 \begin{split}
 (2\pi)^2 \int_{E}  &\frac{\mA_{\metric_E}^n\wedge (\frac{n\cst_{[\omega]}}{n+1}\mA_{\metric_E} - \iota_{E}^*(\rho^{\metric} - \delta +1))}{\delta_E +1} \\  &= 2\pi \int_{\bu}  \delta_\bu \wedge \frac{\mA_{\metric_\bu}^n\wedge (\frac{n\cst_{[\omega]}}{n+1}\mA_{\metric_\bu} - \iota_{\bu}^*(\rho^{\metric} - \delta +1))}{\delta_\bu +1}.  
 \end{split}\end{equation}
%
 To get the whole Donaldson--Futaki invariant we still have to add the localization over $Z$ which is 
 \begin{equation} 2\pi \int_{\bu}   \frac{\mA_{\metric_\bu}^n\wedge \left(\frac{n\cst_{[\omega]}}{n+1}\mA_{\metric_\bu} - \iota_{\bu}^*\rho^{\metric}\right)}{\delta_\bu +1}
 \end{equation}  using \eqref{eqEulerclassesDEFORM} and that the weight of the action on $N_\bu^\mX$ is $1$. Therefore, the Donaldson--Futaki invariant of $(\mX,[\metric],\phi,\pi)$ is 
 \begin{equation}\label{eqdeformDFlast} -2\pi \mF_{\metric_{\mE}}(J\partial \theta) + 2\pi \int_{\bu} \mA_{\metric_\bu}^n\wedge \left(\frac{n\cst_{[\omega]}}{n+1}\mA_{\metric_\bu} - \iota_{\bu}^*\rho^{\metric}\right) + 2\pi \int_{\bu}   \frac{\mA_{\metric_\bu}^n\wedge (\delta_\bu^2-\delta_\bu)}{\delta_\bu +1}. 
 \end{equation}
 
 Developping the middle term with $\mA_{\metric_\bu} = \metric_\bu -\mom_\bu$ and $[\iota_{\bu}^*\rho^{\metric}] = [\rho^{\metric_\bu}] -\delta_Z$, we get \eqref{eqDFdeformCONE}.

\bibliographystyle{abbrv}

%
%
%
%
%
%
%
%
%
%
%

\end{document}